\documentclass[10pt,hidelinks]{siamart220329}
\usepackage[utf8]{inputenc}
\usepackage[utf8]{inputenc}
\usepackage[T1]{fontenc}
\usepackage{amsmath}
\usepackage{amssymb}
\usepackage{capt-of}
\usepackage{hyperref}
\usepackage{amssymb}
\usepackage{amsmath}
\usepackage[capitalize]{cleveref}
\usepackage{cancel}
\usepackage{color}
\usepackage{booktabs}
\usepackage{longtable}
\usepackage{xcolor}
\usepackage{scrextend}
\author{Karthik Chikmagalur\thanks{Department of Mechanical Engineering, University of California Santa Barbara(\email{kchikmagalur@ucsb.edu})} \and Bassam Bamieh\thanks{Department of Mechanical Engineering, University of California Santa Barbara(\email{bamieh@ucsb.edu})}}
\DeclareMathOperator{\tr}{Tr}
\usepackage{algorithm}
\usepackage{algorithmic}
\usepackage{tikz}
\usetikzlibrary{shapes, arrows, positioning, calc}
\definecolor{darkgreen}{HTML}{005f5f}
\definecolor{magenta}{HTML}{A21E63}
\definecolor{navyblue}{HTML}{2E60AA}

\usepackage{capt-of}

\usepackage{amsmath}
\usepackage{amssymb}

\usepackage{graphicx}

\date{}
\title{An Implicit Function Method for Computing the Stability Boundaries of Hill's Equation}
\providecolor{url}{HTML}{0077bb}
\providecolor{link}{HTML}{882255}
\providecolor{cite}{HTML}{999933}
\hypersetup{
  pdfauthor={},
  pdftitle={An Implicit Function Method for Computing the Stability Boundaries of Hill's Equation},
  pdfkeywords={},
  pdfsubject={},
  pdfcreator={Emacs 29.3 (Org mode 9.7-pre)},
  pdflang={English},
  breaklinks=true,
  colorlinks=true,
  linkcolor=link,
  urlcolor=url,
  citecolor=cite}
\NewCommandCopy{\oldtoc}{\tableofcontents}
\renewcommand{\tableofcontents}{\begingroup\hypersetup{hidelinks}\oldtoc\endgroup}
\begin{document}

\maketitle
\begin{abstract}
Hill's equation is a common model of a time-periodic system that can undergo parametric resonance for certain choices of system parameters.  For most kinds of parametric forcing, stable regions in its two-dimensional parameter space need to be identified numerically, typically by applying a matrix trace criterion.  By integrating ODEs derived from the stability criterion, we present an alternative, more accurate and computationally efficient numerical method for determining the stability boundaries of Hill's equation in parameter space.  This method works similarly to determine stability boundaries for the closely related problem of vibrational stabilization of the linearized Katpiza pendulum.  Additionally, we derive a stability criterion for the damped Hill's equation in terms of a matrix trace criterion on an equivalent undamped system.  In doing so we generalize the method of this paper to compute stability boundaries for parametric resonance in the presence of damping.
\end{abstract}


\section{Introduction}
\label{sec:orga9d7ffe}
\label{sec:introduction}

Second order oscillators with time-periodically varying system parameters can experience the related phenomena of parametric resonance and vibrational stabilization.  Parametric resonance involves exponential amplitude growth under certain combinations of the oscillator's ostensible ``natural frequency'' and parametric forcing amplitude.  Parametric forcing refers to the effect of a time-periodic system parameter that (typically) multiplies the state, and differs from the more typical additive harmonic forcing in this way.  Parametric resonance in observed in a variety of natural and engineered systems, such as in hydrodynamic instabilities~\cite{kelly1965,benjamin1954}, time-periodic axial loading of columns~\cite{iwatsubo1974} and in MEMS devices subjected to alternating voltage~\cite{caruntu2014}.

A canonical model for linear second order systems that can undergo parametric resonance is Hill's equation~\cite{magnus2004Hill}, a harmonic oscillator with a time-periodically varying spring constant.  Certain combinations of the system's natural frequency and parametric forcing amplitude can cause instability, even at arbitrarily low forcing amplitudes under low or zero damping conditions.  These unstable parameter ranges form the well-known Arnold tongue structures~\cite{nayfeh1995}.

Vibrational stabilization, the flip side of parametric resonance, involves the stabilization via time-periodic parametric forcing of nominally unstable systems, typically a second order oscillator with a negative spring constant.  A common example of this is the Kapitza pendulum, an inverted pendulum stabilized via sinusoidal base motion that acts as sensorless feedback.  Although previously considered a high-frequency phenomenon, recent work by Berg~\cite{Berg_2015} shows that vibrational stabilization can be achieved via parametric forcing at lower frequencies with precisely chosen amplitudes.

The stability boundaries in parameter space of both Hill's equation and the linearized Kapitza pendulum are determined by specific contours of an implicit function of the system parameters.  This provides a two-dimensional \emph{stability map} for both kinds of systems.  Such maps have been used in the design of quadrupole ion traps for mass spectrometry~\cite{lee2003Stability}.  While canonical models of both phenomena are typically of second order, higher order systems can be modeled as collections of Hill oscillators~\cite{chikmagalur2024Parametric} or their unstable counterparts~\cite{chikmagalur2024Vibrational} under certain conditions, and stability criteria can be found as the composition or overlap of the individual two-dimensional stability maps of the component systems.

It is possible to derive analytical criteria for parametric resonance in the limit of small parametric forcing~\cite{nayfeh1995} using the classical Floquet theory~\cite{Berg_2015}, and for vibrational stabilization using higher-order averaging methods~\cite{maggia2019}.  However, these results are valid in limited regimes -- respectively, in the limit of low parametric forcing amplitudes and frequencies for parametric resonance and vibrational stabilization.  More general estimates require numerical computation.  Typically these boundaries are computed and visualized via contour plots of the system, as computed using a marching squares algorithm~\cite{maple2003Geometrica}.  In this work we propose a more efficient method involving direct numerical calculation and integration of this implicit function.  This method leans on the idea that we are only interested in specific contours and not the overall picture, reducing the numerical complexity of the problem.

To illustrate the method, we cover the general idea of integrating an implicit function of two variables in \cref{sec:general-method}.  For Hill's equation this function does not have a closed form, so we cover some relevant properties in \cref{sec:hill-equation-properties}, additional solution steps required in \cref{sec:hill-equation-stability-details}, and provide some examples of the results of this method.  \cref{sec:hill-equation-damped-stability-details} extends the numerical method to the case of Hill's equation with damping.  Finally, we cover some considerations for numerical integration in \cref{sec:numerical-considerations}.
\section{The Implicit Function Method for Contour Plots}
\label{sec:orgd16cde4}
\label{sec:general-method}

The general idea of the method is as follows: given a smooth function \(f(a, \epsilon)\) of variables \(a\), \(\epsilon\), we wish to plot a specific contour \(f(a, \epsilon) = c\) in \((a, \epsilon)\) space passing through \((a_0, \epsilon_0)\), where \(f(a_0, \epsilon_0) = c\). 

The condition \(f(a, \epsilon) = c\) is an implicit definition of the curve \(a(\epsilon)\) or \(\epsilon(a)\) that we seek, and we can find differential equations for them:
\begin{align}
\begin{array}{l@{\quad}c@{\quad}l}
f(a(\epsilon), \epsilon) = c & & f(a, \epsilon(a)) = c \\
\frac{\partial f}{\partial a} \frac{\mathrm{d} a}{\mathrm{d} \epsilon} + \frac{\partial f}{\partial \epsilon} = 0, & \text{or} & \frac{\partial f}{\partial a} + \frac{\partial f}{\partial \epsilon} \frac{\mathrm{d} \epsilon}{\mathrm{d} a} = 0 \\
a(\epsilon_0) = a_0 & & \epsilon(a_0) = \epsilon_0
\end{array}
\end{align}
Where at least one of the two equations is well-defined in a neighborhood of \((a_0, \epsilon_0)\).   We note that neither the ODEs nor the initial conditions involve \(c\) -- these ODEs are satisfied by the contour of \(f\) that passes through \((a_0, \epsilon_0)\) irrespective of its value.

Without loss of generality, we assume \(\frac{\partial f}{\partial a} \ne 0\) in some neighborhood of \((a_0, \epsilon_0)\), so that that the curve \(a(\epsilon)\) is locally well defined.  Then we may write
\begin{align}
\frac{\mathrm{d} a}{\mathrm{d} \epsilon} = - \frac{\frac{\partial f}{\partial \epsilon} }{\frac{\partial f}{\partial a} },\quad a(\epsilon_0) = a_0 \label{eq:ode-for-general-method}
\end{align}
Since \(f\) is known, this ODE can be readily solved to the required precision.  During numerical computation, the case of \(\frac{\partial f}{\partial a} \to 0\) can be identified by large changes in \(\frac{\mathrm{d} a}{\mathrm{d} \epsilon}\) (if using fixed time steps), by stiffness checks or progressively smaller time steps (if using adaptive time stepping), and the reciprocal problem can be solved instead.
\section{Application to Hill's Equation}
\label{sec:org880016b}
\label{sec:application-to-hill-ode}

We apply this method to finding the stability boundaries in parameter space for Hill's equation
\begin{align}
& \ddot{\theta}  + (a + \epsilon p(t)) \theta = 0 \label{eq:hill-ode} \\
& p(t) =  p(t +\nonumber 2\pi), 
\end{align}
which represents a harmonic oscillator with a periodically varying spring constant.  \(a > 0\) is the mean value of the spring constant, and \(\epsilon\) is the parametric forcing amplitude.  The parametric forcing \(p(t)\) has zero mean.  Such systems can experience parametric resonance for specific choices of \(a\) and \(\epsilon\).  Allowing \(a\) to be negative gives us the linearized model of the Kapitza pendulum, an inverted pendulum with sinusoidal vertical base motion.  This nominally unstable system is vibrationally stabilized for specific choices of \(a\) and \(\epsilon\).  Our treatment of stability boundaries covers both regimes of behavior.
\subsection{Properties of Hill's Equation}
\label{sec:orgb3a906e}
\label{sec:hill-equation-properties}

We review a few salient properties of Hill's equation for applying the implicit function method to compute its parametric stability.

In first order form, \cref{eq:hill-ode} has a Hamiltonian generator \(A(t;a, \epsilon)\) and thus a symplectic state-transition matrix \(\Theta(t,0; a, \epsilon)\) (henceforth just \(A(t)\) and \(\Theta(t,0)\)):
\begin{align}
\frac{\mathrm{d} }{\mathrm{d} t} \begin{bmatrix} \theta \\ \dot{\theta} \end{bmatrix} & = \underbrace{\begin{bmatrix}
0 & 1 \\
\text{-}a\, \text{-}\, \epsilon\, p(t) & 0
 \end{bmatrix}}_{A(t; a, \epsilon)} \begin{bmatrix} \theta \\ \dot{\theta} \end{bmatrix} \label{eq:hill-ode-1st-order}\\
 \begin{bmatrix} \theta(t) \\ \dot{\theta}(t) \end{bmatrix} & = \Theta(t,0; a, \epsilon) \begin{bmatrix} \theta(0) \\ \dot{\theta}(0) \end{bmatrix}
\end{align}
The stability of this linear, time periodic system depends on the spectral radius of its ``monodromy map'', its state transition matrix evaluated at one forcing period \(\Theta(2\pi,0)\).  The system is stable if its spectral radius is \(1\) or smaller, and unstable if it is larger.

This stability condition can be simplified as follows: \(\Theta(t,0)\) has determinant \(1\) for all \(t\), as it is symplectic.  This can be verified directly as:
\begin{align*}
\frac{1}{ \det \Theta(t,0) }\frac{\mathrm{d} \det \Theta(t,0)}{\mathrm{d} t} & =  \tr \left\{ \dot{\Theta}(t,0) \Theta^{\text{-}1}(t,0) \right\} \\
& = \tr \left\{ A(t) \Theta(t,0) \Theta^{\text{-}1}(t,0) \right\} = \tr \left\{ A(t) \right\} \\
& = 0
\end{align*} 
Since \(\det \Theta(0,0) = \det I = 1\), it is \(1\) for all time.

Together, the conditions \(\det \Theta(t,0) = 1\), \(\bar{\sigma}(\Theta(2\pi,0)) \le 1\) and the fact That \(\Theta(2\pi,0)\) is real and \(2 \times 2\) imply that system~\eqref{eq:hill-ode} is stable iff \(\tr \left\{ \Theta(2\pi,0) \right\} = \pm 2\).

The condition \(\tr{ \Theta(2\pi, 0; a, \epsilon)} = \pm 2\) is a function of two variables whose contours in \((a, \epsilon)\) space with values \(\pm 2\) form the boundary between regions of stability and regions of parametric resonance.  These boundaries in parameter space form the well known Arnold tongue structures associated with parametric resonance (\cref{fig:stability-boundaries-explanation}).

For Hill's equation, the theory of linear time-periodic systems provides many analytical methods to find the positions of the tongues in the limit \(\epsilon \to 0\).  Perturbation analysis, the method of Hill's determinants~\cite{wereley1990thesis} and Floquet multipliers~\cite{nayfeh1995} are three common methods.  The primary result is the following:  In the limit \(\epsilon \to 0^+\), system \eqref{eq:hill-ode} is unstable when \(\sqrt{a}\) is a half-multiple of the forcing frequency \(1\), i.e. when
\begin{align}
a = \frac{n^2}{4},\ n \in \mathbb{N} \label{eq:hill-stability-criterion}
\end{align}
Further, it can be shown~\cite{nayfeh1995} that there are no other values of \(a > 0\) (with \(\epsilon \to 0^+\)) for which the system is unstable.
\subsection{Stability Boundary Computation for Hill's Equation}
\label{sec:orgefa48e1}
\label{sec:hill-equation-stability-details}

The stability boundaries are typically computed in the following way:
\begin{enumerate}
\item Create a grid in \((a, \epsilon)\) space that covers the parameter ranges of interest.
\item Find the monodromy map (and thus its trace) numerically at all points on this grid.
\item Plot the contours of the map that have values \(\pm 2\).
\end{enumerate}

This method has three shortcomings.  The first is the computational burden of calculating the value of \(\tr \left\{ \Theta(2\pi, 0) \right\}\) at all grid points.  Essentially, this involves scanning a two dimensional space in search of a one-dimensional curve.  The second is that the slope of the curve is not known apriori, so the spatial grid resolution cannot be adjusted to take advantage of this.  As a result, several runs with progressively finer grids are required to obtain a suitably detailed picture.  Finally, there are regions in the parameter space where \(\tr \left\{ \Theta(2\pi, 0) \right\}\) changes rapidly with \(a\), making it difficult to resolve fine details in these areas in a contour map created with a typical Marching Squares algorithm.

Applying the implicit function method outlined in \cref{sec:general-method} avoids all of these shortcomings, but it requires knowledge of at least one point on the contour of interest.  Specifically, for each contour such that \(\tr \left\{ \Theta(2\pi, 0; a_0, \epsilon_0) \right\} = \pm 2\), we need to know the corresponding \((a_0, \epsilon_0)\).  From \cref{eq:hill-stability-criterion}, this point is \((n^2/4, 0)\).

\begin{figure}[ht]
\centering
\includegraphics[width=0.7\columnwidth]{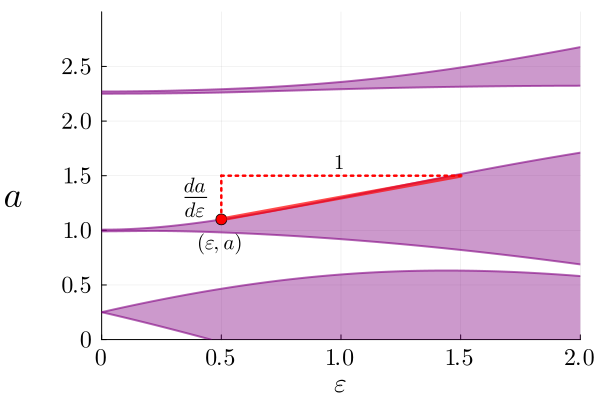}
\caption{\label{fig:stability-boundaries-explanation}A representative stability diagram for Hill's equation illustrating the implicit function method~\eqref{eq:ode-d-a-d-eps} of finding stability boundaries.  The time-periodic forcing in this example is sinusoidal with period \(2\pi\).  The system experiences parametric resonance with paramters in the shaded regions.  All stability boundaries are the curves \((\epsilon, a(\epsilon))\) given by the implicit equation \(\left| \tr \Theta(2\pi, 0; a(\epsilon), \epsilon) \right| = 2\), starting at \(a = n^2/4,\ n \in \mathbb{N}\) and \(\epsilon = 0\).  These curves are found by integrating \cref{eq:ode-d-a-d-eps} starting at these initial conditions.}
\end{figure}

We can now apply the above method to find the stability boundaries for Hill's equation.  To find \(a(\epsilon)\) such that
\begin{align*}
& f(a(\epsilon), \epsilon) := \tr \left\{ \Theta(2\pi, 0; a(\epsilon), \epsilon) \right\} = \pm 2 \text{ and} \\
& a(0) = \frac{n^2}{4}, n \in \mathbb{N},
\end{align*}
we formulate the equivalent of \cref{eq:ode-for-general-method} for this problem:
\begin{align}
\frac{\mathrm{d} a}{\mathrm{d} \epsilon} & = - \frac{\frac{\partial f}{\partial \epsilon} }{\frac{\partial f}{\partial a} } = - \frac{\frac{\partial }{\partial \epsilon} \tr[\Theta(2\pi,0; a, \epsilon)]}{\frac{\partial }{\partial a} \tr[\Theta(2\pi,0; a, \epsilon)]} \nonumber\\
& = - \frac{\tr\left[ \frac{\partial \Theta}{\partial \epsilon}(2\pi,0; a, \epsilon) \right]}{\tr \left[ \frac{\partial \Theta}{\partial a}(2\pi,0; a, \epsilon) \right]}. \label{eq:ode-d-a-d-eps}
\end{align}
The slope \(\frac{\mathrm{d} a}{\mathrm{d} \epsilon}\) of a stability map boundary in \((a(\epsilon), \epsilon)\) space is illustrated in  \cref{fig:stability-boundaries-explanation}.  To calculate \(\frac{\partial \Theta}{\partial \epsilon}\) and \(\frac{\partial \Theta}{\partial a}\), we can use Hill's equation~\eqref{eq:hill-ode-1st-order}.  Differentiating the original system with respect to the parameters \(a\) and \(\epsilon\):
\begin{align}
\dot{\Theta}&(t,0)  = A(t)\, \Theta(t,0) \label{eq:phi-dot}\\
\implies \frac{\partial \dot{\Theta}(t,0)}{\partial a} & = A(t) \frac{\partial \Theta(t,0)}{\partial a} + \frac{\partial A(t)}{\partial a} \Theta(t,0) \label{eq:ode-d-phi-d-a}\\
\implies \frac{\partial \dot{\Theta}(t,0)}{\partial \epsilon} & = A(t) \frac{\partial \Theta(t,0)}{\partial \epsilon} + \frac{\partial A(t)}{\partial \epsilon} \Theta(t,0), \label{eq:ode-d-phi-d-eps}
\end{align}
Where we have suppressed the explicit dependence of \(\Theta\) and \(A\) on \(a\) and \(\epsilon\) in the notation.  The linear ODEs~\eqref{eq:ode-d-phi-d-a} and~\eqref{eq:ode-d-phi-d-eps} can be integrated from \(t = 0\) to \(2\pi\) to give \(\frac{\partial \Theta}{\partial a}(2\pi, 0; a, \epsilon)\) and \(\frac{\partial \Theta}{\partial \epsilon}(2\pi,0;a,\epsilon)\) respectively.  Since \(\Theta(0,0;a,\epsilon) = I\) uniformly for any choice of \(a\) and \(\epsilon\), the initial conditions are uniformly zero in the above equations, and they are driven by the state transition matrix \(\Theta(t,0; a, \epsilon)\) at \((a, \epsilon)\), which appears as an additive input.  This recipe for finding stability boundaries is illustrated in \cref{fig:solution-algorithm-step} and \cref{fig:plot-boundaries-explanation}.

\begin{figure}
\begin{center}
\begin{tikzpicture}[
    auto, node distance = 1.5em,
    block/.style = {
        draw, rectangle,
        minimum height=3em, minimum width=6em,
        inner sep=0.5em
    },
    line/.style = {draw, -latex'},
    >=latex,
    every path/.style={thick}
]

\node (mainblock) [block, anchor=east] {\( \displaystyle \frac{\mathrm{d} a}{\mathrm{d} \epsilon}  =  - \frac{ \tr \left[  {\color{navyblue} \frac{\partial \Theta}{\partial \epsilon}}(2\pi,0; a, \epsilon)  \right] }{ \tr \left[  {\color{magenta} \frac{\partial \Theta}{\partial a}}(2\pi,0; a, \epsilon)  \right]} \)};

\node (auxblock1) [block, below=6em of mainblock] {\( \displaystyle \frac{\mathrm{d} }{\mathrm{d} t}{\color{magenta} \frac{\partial \Theta}{\partial a}}  = A {\color{magenta} \frac{\partial \Theta}{\partial a}} + \frac{\partial A}{\partial a} {\color{darkgreen} \Theta} \)};

\node (statetransition)  [block, above=of auxblock1.north west, anchor = south west] {\( \displaystyle \frac{\mathrm{d} }{\mathrm{d} t}{\color{darkgreen} \Theta} = A {\color{darkgreen} \Theta} \)};

\node (auxblock2) [block, below=of auxblock1] {\( \displaystyle \frac{\mathrm{d} }{\mathrm{d} t} {\color{navyblue} \frac{\partial \Theta}{\partial \epsilon}}  = A {\color{navyblue} \frac{\partial \Theta}{\partial \epsilon}} + \frac{\partial A}{\partial \epsilon} {\color{darkgreen} \Theta} \)};

\node (rightmost) [coordinate, right=7em of auxblock1] {};
\node (leftmost)  [coordinate, left=5em of auxblock1, yshift=-0.5em] {};
\node (statein) [coordinate, left=5em of statetransition] {};
\node (stateright) [coordinate, right=5em of statetransition] {};
\node (stateleft) [coordinate, left=3em of statetransition, yshift=-2.25em] {};

\draw [color=gray] (statetransition) to node[pos=0.5,above,color=black] {$\displaystyle {\color{darkgreen} \Theta}(t,0;a, \epsilon)$} (stateright) |- (stateleft);
\draw [->] (statein) |- (statetransition);

\draw[->,color=gray] (stateleft) |- ($ (auxblock1.west) + (0, 0.5em) $);
\draw[->,color=gray] (stateleft) |- ($ (auxblock2.west) + (0, 0.5em) $);

\draw [<-] (mainblock) -|  (rightmost) --  node[pos=0.5, above] {\( \displaystyle {\color{magenta}\frac{\partial \Theta}{\partial a}}(2\pi, 0; a, \epsilon)  \)} (auxblock1);
\draw [<-] (rightmost) |- node[pos=0.75, above] {\( \displaystyle {\color{navyblue} \frac{\partial \Theta}{\partial \epsilon}}(2\pi, 0; a, \epsilon) \)} (auxblock2);

\draw [<-] ($(auxblock1.west) + (0,-0.5em)$) -- (leftmost) -- (statein);
\draw [<-] (statein) |- node[pos=0.75, above] {$\displaystyle (a,\epsilon)$} (mainblock);
\draw [<-] ($(auxblock2.west) + (0,-0.5em)$) -| (leftmost);
\end{tikzpicture}
\end{center}
\caption{\label{fig:solution-algorithm-step}Illustration of the solution method for stability boundary calculation for Hill's equation.  The goal is to find the set of points \(\left\{ (a, \epsilon) \right\}\) that satisfies the topmost block, \cref{eq:ode-d-a-d-eps}, which is a curve passing through known \((a_0, \epsilon_0)\).  For Hill's equation this stability boundary is representable as \(a(\epsilon)\).  At each time step, the current curve coordinates \((a, \epsilon)\) determine the state transition matrix \(\Theta(;a, \epsilon)\) that appears as an input elsewhere.  We compute the state transition matrix over one forcing period and use it to integrate systems~\eqref{eq:ode-d-phi-d-a} and \eqref{eq:ode-d-phi-d-eps}.  The final value of \(\frac{\partial \Theta}{\partial a}\) and \(\frac{\partial \Theta}{\partial \epsilon}\) gives us the value of \(\frac{\mathrm{d} a}{\mathrm{d} \epsilon}\) for this time step.  Advancing the main integrator in time in turn gives us the new coordinates \((a,\epsilon)\) required to continue the solution.  (The dependence of \(\Theta(t,0; a, \epsilon)\) and \(A(t)\) on time has been suppressed in the notation.)}
\end{figure}
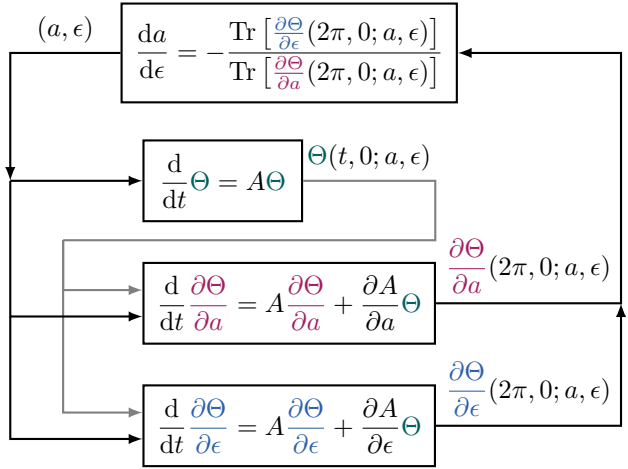

\begin{figure}[htb]
\centering
\includegraphics[width=0.8\textwidth]{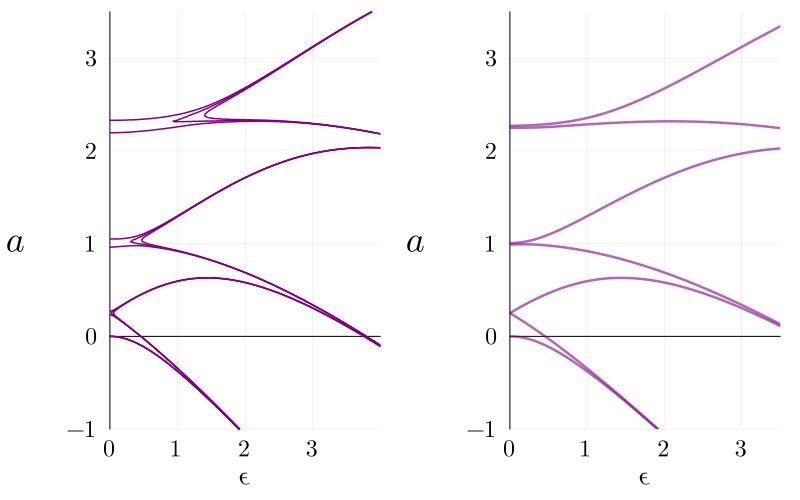}
\caption{\label{fig:boundaries-cos-contour-comparison}Comparison of the stability map in \((\epsilon, a)\) parameter space of Hill's equation with sinusoidal parametric forcing, as computed by (left) the contour method and (right) the implicit function method detailed in \cref{sec:hill-equation-stability-details} and \cref{fig:stability-boundaries-explanation}.  The contour method involves plotting the contour of the function \(f(a, \epsilon) =  \tr\ \Theta(2\pi, 0; a, \epsilon)\) with absolute value \(2\).  This function is computed at a grid resolution of \((\Delta \epsilon, \Delta a) = (0.02, 0.02)\), and the contour levels plotted are \(1.99 < \left| \tr\ \Theta(2\pi, 0; a, \epsilon) \right| < 2.01\).  The implicit function method involves integrating one implicit function for each Arnold tongue boundary \cref{eq:ode-d-a-d-eps} with a step of \(\Delta \epsilon = 0.05\).  The implicit function method is more efficient, and provides much better resolution of the Arnold tongues as well, especially for small \(\epsilon\).  In our tests, it is also faster by a factor of \(10-50\), depending on the desired grid resolution.}
\end{figure}

\cref{fig:boundaries-cos-contour-comparison}  compares the stability maps for Hill's equation as computed using the contour method and the implicit function method introduced in this paper.  The implicit function method provides better resolution with significantly less computational effort.

\cref{fig:stability-diagram-sideways-all} shows the stability boundaries computed using this method for Hill's equation with a few different parametric forcing functions.

\begin{figure*}
\centering
\includegraphics[width=0.95\textwidth]{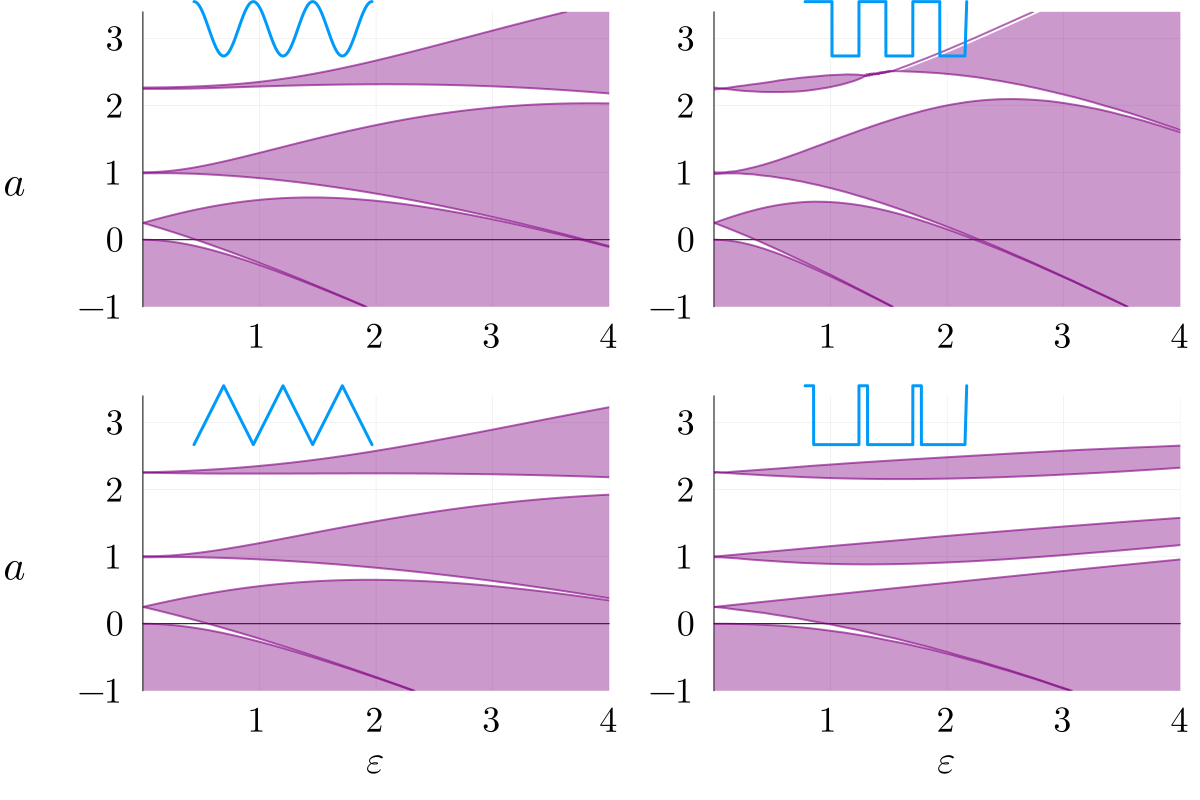}
\caption{\label{fig:stability-diagram-sideways-all}Stability diagrams in the \((\epsilon, a)\) parameter space for Hill's equation, for four different kinds of zero-mean parametric forcing (inset, in blue).  Clockwise from the top, they are sinusoidal, square waves with even and uneven duty cycles, and a periodic ramp.  The stability boundaries are computed using the implicit function method (\cref{fig:solution-algorithm-step}) for the different forcing functions.  For \(a > 0\), the plots show the first three Arnold tongues, whose locations as \(\epsilon \to 0^+\) are independent of the type of forcing.  For \(a < 0\), the plots show the \((\epsilon, a)\) parameter ``window'' that causes vibrational stabilization of the nominally exponentially unstable system.}
\end{figure*}
\subsection{Stability boundaries for Hill's equation with damping}
\label{sec:orga141756}
\label{sec:hill-equation-damped-stability-details}

In this section we apply the implicit function method of~\cref{sec:hill-equation-stability-details} to Hill's equation with viscous damping coefficient \(\kappa\):
\begin{equation}
\ddot{\theta} + 2 \kappa \dot{\theta} + \left( a + \epsilon p(t) \right) \theta = 0. \label{eq:hill-ode-damped}
\end{equation}
Through the transformation \(z(t) := e^{\kappa t} \theta(t)\), it is possible to derive an equivalent stability criterion for a related (undamped) Hill equation.

\begin{theorem}
Given~\cref{eq:hill-ode-damped}, we define the related (undamped) Hill's equation
\begin{align*}
\ddot{z} + \left( a - \kappa^2 + \epsilon p(t) \right) z = 0
\end{align*}
with state transition matrix \(\Phi(t,0; a, \epsilon)\).  Then system~\eqref{eq:hill-ode-damped} is unstable whenever
\begin{align}
\left\lvert \tr \left[ \Phi(2\pi, 0; a, \epsilon) \right] \right\rvert > 2 \cosh(2\pi \kappa) \label{eq:hill-stability-criterion-damped}
\end{align}
\end{theorem}

\begin{proof}
See~\cref{sec:damping-derivation}.
\end{proof}

Since \(2 \cosh(2\pi \kappa) \ge 2\), the Arnold tongues for Hill's equation with damping do not meet on the \(a\) axis, and there is no equivalent to criterion~\eqref{eq:hill-stability-criterion}.  Integrating~\cref{eq:ode-d-a-d-eps} along the contour \(\left| \tr \left[ \Phi(2\pi, 0; a, \epsilon) \right] \right| = 2 \cosh(2\pi \kappa)\) thus requires finding one point along each ``branch'' of the stability boundary.

We find these initial conditions \((\epsilon_0, a_0)\) numerically, using the criteria \(\left| \tr \left[ \Phi(2\pi, 0; a_0, \epsilon_0) \right] \right| =\) \(2 \cosh(2\pi \kappa)\) and \(\left. \frac{\partial }{\partial a} \left[ \Phi(2\pi, 0; a, \epsilon) \right] \right|_{( \epsilon_0, a_0)} = 0\).  The rest of the treatment is identical to the undamped case covered in~\cref{sec:hill-equation-damped-stability-details}.  \cref{fig:plot-boundaries-explanation} illustrates the method.

\begin{figure}[ht]
\centering
\includegraphics[width=1.0\textwidth]{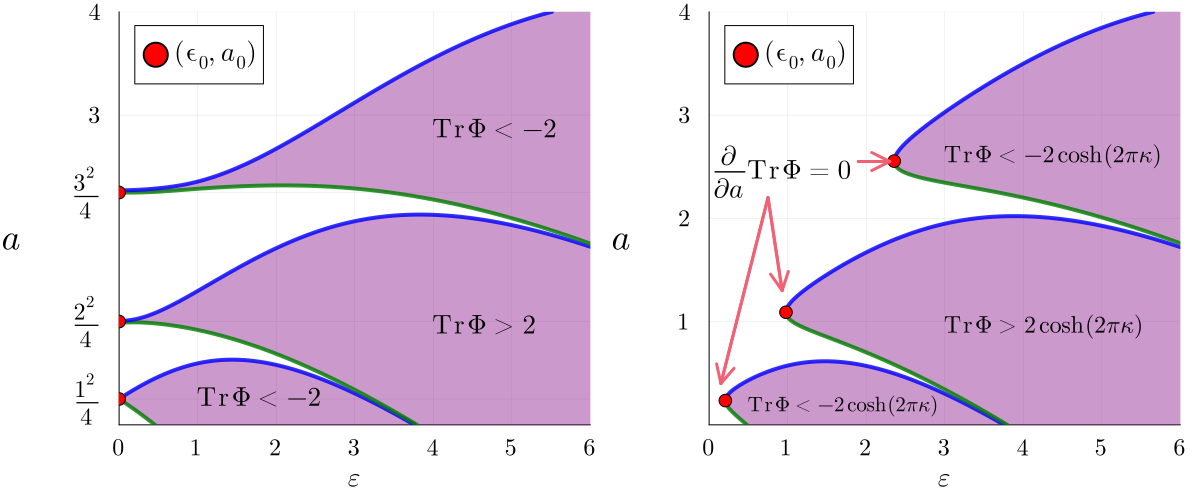}
\caption{\label{fig:plot-boundaries-explanation}Illustration of the method for stability boundary calculations for Hill's equation with and without damping (\cref{eq:hill-ode,eq:hill-ode-damped} respectively).  \emph{Left}: Hill's equation without damping.  Each Arnold tongue boundary corresponds to a contour \(\left| \tr \left[ \Theta(2\pi, 0) \right] \right| = 2\) and meets the \(a\) axis at \(n^2/4\), \(n \in \mathbb{N}\).  Thus ODE~\eqref{eq:ode-d-a-d-eps} can be integrated starting at \((\epsilon_0, a_0) = \left( 0, n^2/4 \right)\).  See \cref{fig:solution-algorithm-step} for details of the integration process.  \emph{Right}: Hill's equation with damping coefficient \(\kappa\).  The method is similar to the undamped case, but the initial conditions \((\epsilon_0, a_0)\) for each Arnold tongue have to be found numerically.  The ODE~\eqref{eq:ode-d-a-d-eps} is integrated from the left-most point of each tongue, corresponding to \(\left. \tr \left[ \Phi(2\pi, 0) \right] \right|_{(\epsilon_0, a_0)} = 2 \cosh(2\pi \kappa)\) and \(\left. \frac{\partial }{\partial a} \tr \left[ \Phi(2\pi, 0) \right] \right|_{(\epsilon_0, a_0)} = 0\).  See \cref{sec:hill-equation-damped-stability-details} for details.  In both cases, finding \(a\) as a function of \(\epsilon\) requires picking one of the two ``branches'' of the stability boundary passing through \((\epsilon_0, a_0)\).  The upper and lower branches of the stability boundary (colored blue and green respectively) are resolved by perturbing \((\epsilon_0, a_0)\) appropriately.}
\end{figure}
\section{Considerations for Numerical Integration}
\label{sec:org8570350}
\label{sec:numerical-considerations}

In this section we briefly mention some details of and considerations for the numerical integration of \cref{eq:ode-d-a-d-eps}.

\begin{itemize}
\item Since the Hill equation is a Hamiltonian dynamical system,  \cref{eq:phi-dot} needs to be solved using a symplectic integrator.  Depending on the required accuracy, a symplectic Euler method or an ``almost symplectic'' trapezoidal method can suffice.

\item The dynamics of \(\frac{\partial \Theta}{\partial a}\), \(\frac{\partial \Theta}{\partial \epsilon}\) and \(\frac{\mathrm{d} a}{\mathrm{d} \epsilon}\) (\cref{eq:ode-d-phi-d-a,eq:ode-d-phi-d-eps,eq:ode-d-a-d-eps} respectively) have no special structure and are solved for using an explicit Runge-Kutta (Owren-Zennaro optimized interpolatioon 5/4~\cite{owren1992Derivation}) method.

\item Since \(\tr \left[ \frac{\partial \Theta}{\partial a}(2\pi, 0; a, \epsilon) \right] \to 0\) where two Arnold tongues meet, the integration for \cref{eq:ode-d-a-d-eps} can stall around these points.  We switch conditionally to solving the reciprocal problem when appropriate:
\begin{align*}
\frac{\mathrm{d} \epsilon}{\mathrm{d} a} = - \frac{\tr \left[ \frac{\partial \Theta}{\partial a}(2\pi, 0; a, \epsilon) \right]}{\tr \left[ \frac{\partial \Theta}{\partial \epsilon}(2\pi, 0; a, \epsilon) \right]}
\end{align*}

\item Finally, computing the stability map for the damped Hill's equation requires finding points \((\epsilon_0, a_0)\) where \(\tr \left[ \Phi(2\pi, 0; a_0, \epsilon_0) \right] = 2 \cosh(2\pi \kappa)\) and \(\left. \frac{\partial }{\partial a} \tr \left[ \Phi(2\pi, 0; a, \epsilon) \right] \right|_{(\epsilon_0, a_0)} = 0\).  These are found by integrating \cref{eq:phi-dot,eq:ode-d-phi-d-a} and a binary search above the locations of the tongues of the undamped system, i.e. in the vicinity of \((\epsilon_0, a_0) = \left( 0, n^2/4 \right)\) for \(n \in \mathbb{N}\).
\end{itemize}
\section{Extensions and Limitations of the Implicit Function method}
\label{sec:org204cc91}
\label{sec:conclusion}

The implicit function method detailed in \cref{sec:general-method} can be applied to compute specific contours of any smooth scalar-valued function more efficiently than evaluating the function on a grid.

While the method provides better resolution of the tongues compared to plotting contours of the trace of the monodromy map, it also rests on certain assumptions about the geometry of the Arnold tongues, for instance that the stability boundaries are smooth and representable in the form \(a(\epsilon)\).  It is thus not a true parametric description of the stability boundaries of the form \((a(\eta), \epsilon(\eta))\) for some parameter \(\eta\), and requires careful handling of ``crossover'' points of the Arnold tongues, where two tongues meet.

It is also reliant on the existence of a scalar criterion for stability, of the form \(f(a,\epsilon) = \text{constant}\).  The stability criterion for parametric resonance in higher-order linear systems involves the eigenvalues of the monodromy map that is not reducible to a criterion involving its trace.  When the system cannot be diagonalized into a collection of decoupled second order time-periodic oscillators,  neither the contour method nor the implicit function method discussed in this paper can be directly applied to obtain stability maps.  Extending the method to generate the stability maps of higher order systems, possibly using the symplectic properties of the monodromy map, is a promising avenue of future work.

\bibliographystyle{siamplain}
\bibliography{~/Documents/roam/biblio,/home/karthik/Documents/roam/biblio}

\appendix
\section{Hill's equation with damping}
\label{sec:org4d7939c}
\label{sec:damping-derivation}

In this section we review properties of the monodromy map of the damped Hill's equation that we use in \cref{sec:hill-equation-damped-stability-details}.

Hill's equation with damping is (see \cref{eq:hill-ode-damped})
\begin{equation}
\ddot{\theta} + 2 \kappa \dot{\theta} + \left( a + \epsilon p(t) \right) \theta = 0
\end{equation}
The state-transition matrix for this system is given, as before, by \(\Theta(t,0; a, \epsilon, \kappa)\), henceforth just \(\Theta(t,0)\).  This can be transformed into the undamped Hill equation through the variable substitution \(z(t) := e^{\kappa t} \theta(t)\):
\begin{align}
\begin{bmatrix} z(t) \\ \dot{z}(t) \end{bmatrix} & = e^{\kappa t} \underbrace{\begin{bmatrix}
1 & 0 \\
\kappa & 1
 \end{bmatrix}}_K \begin{bmatrix} \theta(t) \\ \dot{\theta}(t) \end{bmatrix} =: e^{\kappa t} K \begin{bmatrix} \theta(t) \\ \dot{\theta}(t) \end{bmatrix} \label{eq:hill-equation-damping-transformation}\\
\implies & \ddot{z} + \left( a - \kappa^2 + \epsilon p(t) \right) z = 0, \label{eq:hill-equation-damped-transformed}
\end{align}
where we have defined the transformation matrix
\begin{align}
K := \begin{bmatrix} 1 & 0 \\ \kappa & 1 \end{bmatrix}. \label{eq:damping-transformation-matrix}
\end{align}
The nominal spring constant in the transformed coordinates is lower than the original by \(\kappa^2\).  We denote the state transition matrix for system~\eqref{eq:hill-equation-damped-transformed} as \(\Phi(t,0)\), suppressing the dependence on the parameters \(a\), \(\epsilon\) and \(\kappa\).

System~\eqref{eq:hill-ode-damped} is unstable when \(\Theta(2\pi, 0)\) has an eigenvalue outside the unit circle.  Our goal is to obtain an equivalent stability criterion for the damped system \cref{eq:hill-ode-damped} in terms of the monodromy map \(\Phi(2\pi, 0)\) of the transformed system.  The implicit function method of \cref{sec:application-to-hill-ode} can then be applied to the transformed, undamped system to plot the stability map of the original damped system.  To do this we use the following property of the solution of the damped Hill's equation:

\begin{lemma}
The monodromy map \(\Theta(2\pi, 0)\) of the damped Hill equation~\eqref{eq:hill-ode-damped} is a scalar multiple of a symplectic matrix.  Specifically, \(e^{2\pi \kappa} \Theta(2\pi, 0)\) is symplectic.
\end{lemma}

\begin{proof}
From \cref{eq:hill-equation-damping-transformation}, we can express the monodromy map \(\Phi(2\pi, 0)\) in terms of that of the original damped system:
\begin{align*}
\begin{bmatrix} z(2\pi) \\ \dot{z}(2\pi) \end{bmatrix} & = e^{2\pi \kappa} K \begin{bmatrix} \theta(2\pi) \\ \dot{\theta}(2\pi) \end{bmatrix} \\
& = e^{2\pi \kappa} K\, \Theta(2\pi, 0)\, \begin{bmatrix} \theta(0) \\ \dot{\theta}(0) \end{bmatrix} \\
& = \underbrace{e^{2\pi \kappa} K\, \Theta(2\pi, 0)\, K^{\text{-}1}}_{\Phi(2\pi, 0)} \begin{bmatrix} z(0) \\ \dot{z}(0) \end{bmatrix},
\end{align*} 
so that \(\Phi(2\pi, 0) = e^{2\pi \kappa} K \Theta(2\pi, 0) K^{\text{-}1}\).  We note that \(K\) is \(2 \times 2\) and \(\det K = 1\), so it is symplectic as well.  Thus
\begin{align}
e^{2\pi \kappa} \Theta(2\pi, 0) = K^{\text{-}1} \Phi(2\pi, 0) K \label{eq:hill-ode-damped-monodromy-transform}
\end{align}
is a product of three symplectic matrices and is also symplectic.
\end{proof}

When \(\Theta(2\pi, 0)\) has an eigenvalue at \(\lambda = \pm1\), \(e^{2\pi \kappa} \Theta(2\pi, 0)\) has an eigenvalue at \(e^{2\pi \kappa} \lambda = \pm e^{2\pi \kappa}\).  Since the latter matrix has determinant \(1\), its other eigenvalue is at \(\pm e^{\text{-}2\pi \kappa}\).  Next, we observe that \(e^{2\pi \kappa} \Theta(2\pi, 0)\) and \(\Phi(2\pi, 0)\) are related by a similarity transform \cref{eq:damping-transformation-matrix}, so their eigenvalues and matrix traces are the same.  The instability criterion for \(\theta(t)\) in terms of the transformed~\cref{eq:hill-equation-damped-transformed} is thus
\begin{align}
\left| \tr \left[ \Phi(2\pi, 0) \right] \right| > e^{2\pi \kappa} + e^{\text{-} 2\pi \kappa} = 2 \cosh(2\pi \kappa)
\end{align}
This result is summarized in \cref{tab:damping-transformation-comparison-table}.  We note that \(e^{2\pi \kappa} + e^{\text{-}2\pi \kappa} \ge 2\), with the lower bound achieved in the undamped case of \(\kappa = 0\).

\begin{table}[tb] \label{tab:damping-transformation-comparison-table}
\caption{A summary of the stability criterion for the damped Hill's equation (\cref{eq:hill-ode-damped}) in its original and transformed (\cref{eq:hill-equation-damped-transformed}) form.  Both criteria involve the traces of the monodromy maps of the the systems, and the stability boundary corresponds to different contours.  The two forms are equivalent.}
\begin{tabular}{lcc}
 Hill's equation: & Damped & Transformed (undamped) \\
\toprule
  System & \(  \ddot{\theta} + 2 \kappa \dot{\theta} + \left( a + \epsilon p(t) \right) \theta = 0  \) & \(  \ddot{z} + \left( a - \kappa^2 + \epsilon p(t) \right) z = 0  \) \\
  State & \( \left[  \begin{smallmatrix} \theta & \dot{\theta} \end{smallmatrix}^{\mathrm{T}}  \right] \) & \( \left[  \begin{smallmatrix}z & \dot{z}\end{smallmatrix}  \right]^{\mathrm{T}} \) \\
  Monodromy map & \( \Theta(2\pi,0) \) & \( \Phi(2\pi,0) \) \\
  Stability criterion & \( \left| \tr \left[ \Theta(2\pi, 0) \right] \right| \le 2 \) & \( \left| \tr \left[ \Phi(2\pi, 0) \right] \right| \le 2\cosh(2\pi \kappa) \) \\
\bottomrule
\end{tabular}
\end{table}
\end{document}